\documentclass[12pt, reqno]{amsart}
\usepackage{amsmath,amssymb,amsthm}

\usepackage{amsmath, amssymb}
\usepackage{amsthm, amsfonts, mathrsfs}
\usepackage{fullpage}
\usepackage{amsfonts,graphicx}
\numberwithin{equation}{section}
\usepackage[colorlinks=true, pdfstartview=FitV, linkcolor=blue, citecolor=blue, urlcolor=blue]{hyperref}

\newtheorem{theorem}{Theorem}

\newtheorem{corollary}[theorem]{Corollary}

\newtheorem{lemma}[theorem]{Lemma}

\usepackage{graphicx}
\usepackage{textcomp}

\date{}

\begin{document}

\title[stability of  an abstract system with infinite history ]{stability of  an abstract system with infinite history }
\author[ Aberrahmane Youkana]
{ Abderrahmane Youkana}

\address{Abderrahmane Youkana \newline
Laboratory of Mathematical Techniques for Applications, Faculty of Mathematics and Computer Science,
University of Batna2, 05078 Batna,  Algeria.}
\email{abderrahmane.youkana@univ-batna2.dz}
\email{abder.youkana@yahoo.fr}
\maketitle
\begin{abstract}
This work is concerned with stabilization of an abstract linear dissipative integrodiffrential equation with infinite memory modeling linear viscoelasticity where the relaxation function satisfies $ g'(t)\leq -\xi(t)g^p(t), \forall t\geq 0, 1\leq p<\frac{3}{2}$. Our result improves earlier results in the literature.
\end{abstract}
\section{Introduction}
Let us denote by $\mathcal{H}$ a  Hilbert space with inner product and related norm denoted by $\langle,\rangle$ and   $ \Vert. \Vert$ respectively.
Let $A : \mathcal{D}(A) \longrightarrow \mathcal{H}$ and $ B : \mathcal{D}(B) \longrightarrow \mathcal{H} $ be self-adjoint linear positive definite operators with domains  $\mathcal{D}(A) \subset \mathcal{D}(B) \subset \mathcal{H}  $ \ such that the embeddings are dense and compact.\\

We are interested in energy decay of the solution u to the following initial boundary
value problem 
\begin{eqnarray}\label{I1.1}
 u_{tt}  +Au-\int_0^{+\infty} g(s)Bu(t-s) \ ds=0,  \qquad \textrm{$\forall t>0 $},
\end{eqnarray}
with initial conditions
\begin{equation}\label{I1.2}
\left\lbrace
\begin{array}{ll}\bigskip
u(-t)  = u_0(t), \qquad  \forall t \in \mathbb{R}_+, \\\bigskip
u_t(0)=u_1,
\end{array}
\right.
\end{equation}
where $ u_0$ and $u_1$ are given history and initial data, 
g is a positive and nonincreasing function called the relaxation function.\\
\subsection{Well Posedness}
By following the brilliant intuition of Dafermos \cite{Dafer70,Dafer76}, we introduce the
relative history of $u$ defined as
\begin{eqnarray*}
\left\lbrace
\begin{array}{ll}\bigskip
\eta^t(s)=u(t)-u(t-s),   & \forall t,s \in \mathbb{R}_+,\\\bigskip
\eta_0(s)=\eta^0(s)=u_0(0)-u_0(s), & \qquad \forall s \in \mathbb{R}_+.
\end{array}
\right.
\end{eqnarray*}
Equation \eqref{I1.1}-\eqref{I1.2}  can be rewritten as an abstract linear first-order system of the form
\begin{eqnarray}\label{I1.3}
\left\lbrace
\begin{array}{ll}\bigskip
\mathcal{U}_t+ \mathcal{A}\mathcal{U}(t)=0, \qquad \forall t>0,  \\\bigskip
\mathcal{U}(0)=\mathcal{U}_0,
\end{array}
\right.
\end{eqnarray}
where $\mathcal{U}_0=(u_0(0),u_1,\eta_0)^T \in \mathcal{H}= \mathcal{D}(A^{\frac{1}{2}})\times H \times L_g^2(\mathbb{R}_+, \mathcal{D}(B^{\frac{1}{2}})), \ \mathcal{U}=(u,u_t,\eta^t)^T$ and\\
$L^2_g(\mathbb{R}_+,\mathcal{D}(B^{\frac{1}{2}}))$ is the weighted space with respect to the
measure $ g(s) ds $ defined by 
\begin{eqnarray*}
L^2_g(\mathbb{R}_+,\mathcal{D}(B^{\frac{1}{2}}))=\left\lbrace  z: \mathbb{R}_+\longrightarrow \mathcal{D}(B^{\frac{1}{2}}), \
\int_0^{+\infty} g(s) \Vert B^{\frac{1}{2}}z(s) \Vert^2 \ ds< +\infty  \right\rbrace
\end{eqnarray*}
endowed with the inner product
\begin{eqnarray*}
\langle z_1,z_2\rangle_{L^2_g(\mathbb{R}_+,\mathcal{D}(B^{\frac{1}{2}}))}= \int_0^{+\infty } g(s) \langle B^{\frac{1}{2}}z_1(s) ,B^{\frac{1}{2}}z_2(s)\rangle \ ds. 
\end{eqnarray*}
The operator $\mathcal{A}$ is defined by 
\begin{eqnarray*}
\mathcal{A}(v,w,z)^T=\left(-w,Av-g_0Bv+\int_0^{+\infty} g(s)  B z(s) \ ds,  \frac{\partial z}{\partial s}-w  \right)^T,
\end{eqnarray*}
where $ g_0= \int_0^{+\infty} g(s) \ ds$, 
\begin{eqnarray*}
\mathcal{D}(\mathcal{A})=\left\lbrace (v,w,z)^T \in \mathcal{H}, \ v \in \mathcal{D}(A), \ 
w \in \mathcal{D}(A^{\frac{1}{2}})), \  z \in \mathcal{L}_g , \ \int_0^{+\infty} g(s)z(s) \ ds \in \mathcal{D}(B)
 \right\rbrace,  
\end{eqnarray*}
and $ \mathcal{L}_g=\left\lbrace  z \in L^2_g(\mathbb{R}_+,\mathcal{D}(B^{\frac{1}{2}})), \partial_s z \in
L^2_g(\mathbb{R}_+,\mathcal{D}(B^{\frac{1}{2}})), z(0)=0
  \right\rbrace. $\\

As shown in  \cite{Assymp07} for example, under the assumptions $ (\mathbb{H}_1)$ and $ (\mathbb{H}_2)$ below, the space $ \mathcal{H}$ endowed with the inner product	
\begin{eqnarray*}
\langle (v_1,w_1,z_1)^T, (v_2,w_2,z_2)^T \rangle_{\mathcal{H}}=
\langle A^{\frac{1}{2}}v_1, A^{\frac{1}{2}}v_2\rangle -g_0\left\langle B^{\frac{1}{2}} v_1,B^{\frac{1}{2}}v_2 \right\rangle  +\langle w_1, w_2\rangle+\langle z_1,z_2\rangle_{L^2_g(\mathbb{R}_+,\mathcal{D}(B^{\frac{1}{2}}))}
\end{eqnarray*}
is a Hilbert space, $ \mathcal{D}(\mathcal{A}) \subset \mathcal{H}$ with dense embedding, and $\mathcal{A}$ is the infinitesimal generator of a linear contraction  $\mathcal{C}_0$ semigroup on $ \mathcal{H}$. Therefore, the classical semigroup theory implies that (see \cite{Pazy}), for any $\mathcal{U}_0 \subset \mathcal{H}$, the system  \eqref{I1.3} has a
unique weak solution 
\begin{eqnarray*}
\mathcal{U} \in  \mathcal{C}(\mathbb{R}_+,\mathcal{H}).
\end{eqnarray*}
Moreover, if $ \mathcal{U}_0 \in \mathcal{D}(\mathcal{A})$, then the solution of \eqref{I1.3} is classical; that is
\begin{eqnarray*}
\mathcal{U}\in C^1(\mathbb{R}_+,\mathcal{H})\cap C(\mathbb{R}_+, \mathcal{D}(\mathcal{A})).
\end{eqnarray*}
\subsection{Stability}
Problems related to \eqref{I1.1}-\eqref{I1.2} have been studied by many authors and several stability results have been established; see \cite{Ch06,FGP10,GMP05,Pata06,Pata09}. The exponential and polynomial decay of the solutions of equation \eqref{I1.1}-\eqref{I1.2} have been studied in \cite{A11}, where it was assumed that $ (\mathbb{H}_1)$ holds and
\begin{itemize}
\item  $ (A_1)$\  There exists an increasing strictly convex function  $G : \mathbb{R}_+\longrightarrow \mathbb{R}_+$ of class 
$ C^1(\mathbb{R}_+) \cap C^2(]0,+\infty[)$ satisfying
\begin{eqnarray*}
G(0)=G'(0)=0  \qquad \textit{and}  \qquad  \lim_{t\longrightarrow  +\infty} G'(t)=+\infty,
\end{eqnarray*}
such that 
\begin{eqnarray*}
\int_0^{+\infty} \frac{g(s)}{G^{-1}(-g'(s))} \ ds + \sup_{s\in \mathbb{R}_+} \frac{g(s)}{G^{-1}(-g'(s))}<+ \infty.
\end{eqnarray*}
\end{itemize}
The author established a general decay estimate given in term of the convex function G. His result generalizes
the usual exponential and polynomial decay results found in the literature. He considered two cases corresponding to the following two conditions on $A$ and $B$:
\begin{eqnarray}\label{H1.4}
\exists a_2>0: \qquad \Vert A^{\frac{1}{2}}v \Vert^2 \leq a_2 \Vert B^{\frac{1}{2}}v \Vert^2, \quad  \forall v \in \mathcal{D}(A^{\frac{1}{2}}). 
\end{eqnarray}
or 
\begin{eqnarray}\label{H1.5}
\exists a_2>0: \qquad\Vert A^{\frac{1}{2}}v \Vert^2 \leq a_2 \Vert A^{\frac{1}{2}} B^{\frac{1}{2}}v \Vert^2, \quad  \forall v \in \mathcal{D}(A^{\frac{1}{2}}B^{\frac{1}{2}}).
\end{eqnarray}
The study of viscoelastic problem  \eqref{I1.1}-\eqref{I1.2} in the particular case $A=B$ was considered by Guesmia and Messaoudi \cite{A.M14}. The authors considered \eqref{H1.6} below with $p=1$ and extended the decay result known for problems with finite history to those with infinite history. In addition, they improved, in some cases, some decay results obtained earlier in \cite{A11}.\\

Very recently, the authors of  \cite{S17} considered the condition 
\begin{eqnarray}\label{H1.6}
 g'(t) \leq - \xi(t) g^p(t), \qquad \forall t\in \mathbb{R}_+,
 \end{eqnarray}
 where $ \xi$ is a positive and nonincreasing function and $ 1\leq p < \frac{3}{2}$, with the objective of improving the decay rate for problems with finite memory.\\

Condition \eqref{H1.6} gives a better description of the growth of g at infinity  and allows to obtain a precise estimate of the energy that is more general than the $ " $stronger $"$ one ($ \xi $ constant and $ p \in [1,\frac{3}{2}[)$ used in the case of past history control \cite{Timo09, Timo08}. We also refer the reader to some recent researches under the condition \eqref{H1.6} with finite history and viscolelastic term \cite{S17} for related results. The authors proved a general decay rate from which the exponential decay  is only a special case. Moreover, the optimal polynomial decay is easily and directly obtained without restrictive conditions.\\

With the above motivations and inspired by the approach of \cite{S17}, in this paper,  we intend to study the general decay result to  problem \eqref{I1.1}-\eqref{I1.2} under suitable  assumptions on the initial data and the relaxation function $g$. Our main contribution is an enhancement to the results of \cite{A.M14,A11} in a way that our result gives a better rate of decay in the polynomial case.

The plan of this paper is as follows. In Section 2, we present some assumptions, preliminaries and some technical lemmas needed to establish the proofs of our results. Section 3 is devoted to the statement and the proof of our main result.
  \section{Preliminaries}
In this section, we shall present some necessary assumptions and prove some important inequalities that will become useful in later stages.\\
Let us assume that\\
$(\mathbb{H}_1)$ There exist positive constants $a_0$ and $a_1$ such that
\begin{equation*}
a_1  \Vert v \Vert^2 \leq  \Vert B^{\frac{1}{2}} v \Vert^2 \leq  a_0\Vert A^{\frac{1}{2}}v \Vert^2, \quad  \forall v \in \mathcal{D}( A^{\frac{1}{2}}).
\end{equation*}
$(\mathbb{H}_2)$
$g : \mathbb{R}_+\rightarrow \mathbb{R}_+ $  is a differentiable nonincreasing  function satisfying
\begin{eqnarray*}
 0<g_0< \frac{1}{a_0}. 
\end{eqnarray*}
$(\mathbb{H}_3)$ There exists a nonincreasing differentiable function  \ $ \xi : \  \mathbb{R}_+\rightarrow \mathbb{R}_+ $  and $ 1\leq p <\frac{3}{2}$ satisfying  \eqref{H1.6}.\\
Throughout the sequel, we denote by $ C$ a generic positive constant which may vary from line to line.
We start with the follwing lemma.
\begin{lemma}
Let $F$and $h$ be two positive functions, and $\alpha, c_1$  and $c_2$ be three positive constants
such that 
\begin{eqnarray}\label{IL11}
F'(t) \leq - c_1\xi^{\alpha+1}(t) F^{\alpha +1}(t) + c_2 h^{\alpha+1}(t), \quad  \forall t \in \mathbb{R}_+.
\end{eqnarray}
Then, for some constant $C > 0$, we have
\begin{eqnarray} \label{IL12}
F(t) \leq C (1+t)^{-\frac{1}{\alpha}} \xi^{-\frac{\alpha+1}{\alpha}}(t)\left[ 1+  \int_0^t (s+1)^{\frac{1}{\alpha}} \xi^{\frac{\alpha+1}{\alpha}} (s) h^{\alpha+1}(s)\ ds 
 \right], \qquad \forall t \in \mathbb{R}_+. 
\end{eqnarray}
\end{lemma}
\begin{proof}
Multiplying \eqref{IL11} by $ \xi^\beta $, where $ \beta>1$ that will be defined later, we find
\begin{eqnarray}\label{IPL1}
\xi^\beta(t) F'(t) \leq - c_1 \xi^{\alpha+1+\beta}(t) F^{\alpha +1}(t) + c_2\xi^\beta(t) h^{\alpha+1}(t).
\end{eqnarray}
By taking advantage of the fact that $ \xi $ is a nonincreasing function, we get
\begin{eqnarray}\label{IPL2}
\left( \xi^\beta(t) F(t)\right)' \leq - c_1 \xi^{\alpha+1+\beta}(t) F^{\alpha +1}(t) + c_2\xi^\beta(t) h^{\alpha+1}(t).
\end{eqnarray}
Also by noting $ \varphi(t) = \xi^\beta(t) F(t)$, and taking  $ \beta= \frac{\alpha+1}{\alpha}$, we obtain 
\begin{eqnarray}\label{IPL3}
\varphi'(t) \leq - c_1\varphi^{\alpha+1}(t)  + c_2\xi^\beta(t) h^{\alpha+1}(t).
\end{eqnarray} 
Following the same steps as in \cite{Ge}, we then find that
\begin{eqnarray}
\varphi(t) \leq C (1+t)^{-\frac{1}{\alpha}}\left[ 1+  \int_0^t \xi^{\frac{\alpha+1}{\alpha}} (s) h^{\alpha+1}(s) (s+1)^{\frac{1}{\alpha}} \ ds 
 \right]. 
\end{eqnarray}
Therefore  \eqref{IL12} is established.
\end{proof}
\begin{lemma}
Assuming that g satisfies $(\mathbb{H}_2)$ and $(\mathbb{H}_3)$ then
\begin{eqnarray}\label{IL2.1}
\int_0^{+\infty} \xi(t) g^{1-\sigma}(t) \ dt < +\infty, \qquad \forall \sigma<2-p.
\end{eqnarray}
\end{lemma}
\begin{proof}
See \cite{S17}.
\end{proof}

\section{Decay of solutions}
In this section, we aim to investigate the asymptotic behavior of solutions for problem \eqref{I1.1}-\eqref{I1.2}.\\
Note that, for any regular solution $ u $ of the problem \eqref{I1.1}-\eqref{I1.2}, it is straightforward to see that 
\begin{eqnarray}\label{ID3.1}
E'(t)&=& \frac{1}{2} \int_0^{+\infty} g'(s) \Vert B^{\frac{1}{2}} \eta^t(s) \Vert^2 \ ds, \quad  \forall t \in \mathbb{R}_+,
\end{eqnarray}
where
\begin{eqnarray}\label{ID3.2}
E(t)&=& \frac{1}{2} \Vert \mathcal{U}(t) \Vert^2_{\mathcal{H}} \nonumber\\
&=&  \frac{1}{2} \left(  \Vert A^{\frac{1}{2}} u(t) \Vert^2 - g_0 \Vert B^{\frac{1}{2}} u(t) \Vert^2+ \Vert u_t(t) \Vert^2 
+ \int_0^{+\infty} g(s)  \Vert B^{\frac{1}{2}} \eta^t(s) \Vert^2 \ ds   \right). 
\end{eqnarray}
\begin{theorem}
Assume that $(\mathbb{ H}_{1})$, $(\mathbb{ H}_{2})$ and $(\mathbb{ H}_{3})$ hold.\\
\begin{enumerate}
\item  Let $ \mathcal{U}_0 \in H$ and $ \mathcal{U}$ be the solution of $ \eqref{I1.3}$.\\
 If \eqref{H1.4} holds, and if,  
\begin{eqnarray}\label{ID3.3}
\exists m_0>0: \quad \Vert B^{\frac{1}{2}} u_0(s) \Vert \leq m_0, \qquad \forall s>0,
\end{eqnarray}
then there exists a positive  constant $C$ such that, for all $t \in \mathbb{R}_+$,
\begin{equation}\label{ID3.5}
E(t) \leq C (1+t)^{-\frac{1}{2p-2}} \xi^{-\frac{2p-1}{2p-2}}(t)\left[ 1+  \int_0^t  (s+1)^{\frac{1}{2p-2}} \xi^{\frac{2p-1}{2p-2}} (s) h^{2p-1}(s)  \ ds 
 \right],
\end{equation}
where $ h(t)=\xi(t) \int_t^{+\infty}g(s) \ ds$.\\

Moreover, if 
\begin{equation}\label{ID3.6}
\int_{0}^{+\infty }(1+t)^{-\frac{1}{2p-2}} \xi^{-\frac{2p-1}{2p-2}}(t)\left[ 1+  \int_0^t  (s+1)^{\frac{1}{2p-2}} \xi^{\frac{2p-1}{2p-2}} (s) h^{2p-1}(s)  \ ds 
 \right] \ dt <+\infty ,
\end{equation}
then, for all $t\in \mathbb{R}_+$,
\begin{equation}\label{ID3.7}
E(t)\leq C (1+t)^{-\frac{1}{p-1}} \xi^{-\frac{p}{p-1}}(t)\left[ 1+ \int_0^t (s+1)^{\frac{1}{p-1}} \xi^{\frac{p}{p-1}} (s) h^{p}(s)  \ ds 
 \right]. 
\end{equation}
\item  Let $ \mathcal{U}_O \in D(A) \times D(A^\frac{1}{2}) \times L^2_g( \mathbb{R}_+, D(A^{\frac{1}{2}} B^{\frac{1}{2}})$ and
$ \mathcal{U}$ be the solution of $ \eqref{I1.3}$.\\
If  \eqref{H1.5} holds, and if, 
\begin{eqnarray}\label{H1.52}
\exists m_0>0: \quad \Vert A^{\frac{1}{2}} B^{\frac{1}{2}}u_0(s) \Vert \leq m_0, \quad \forall s>0,
\end{eqnarray}
then there exists a positive constant $C$ such that, for all $t \in \mathbb{R}_+$,
\begin{eqnarray}
E(t) \leq C \left(  \frac{E_2(0)+E^{2p-1}(0)+ \int_0^t h^{2p-1}(s) \ ds}{ \int_0^t \xi^{2p-1}(s) \ ds } \right)^{ \frac{1}{2p-1}}, 
\end{eqnarray}
where 
\begin{eqnarray}\label{3.38}
E_2(t)&=& \frac{1}{2}\left( \Vert A u(t) \Vert^2-g_0 \Vert A^{\frac{1}{2}}B^{\frac{1}{2}}u(t) \Vert^2
+ \Vert A^{\frac{1}{2}}u'(t)\Vert^2 \right) + \frac{1}{2} \int_0^{+\infty}g(s)\Vert A^{\frac{1}{2}}B^{\frac{1}{2}}\eta^t(s)  \Vert^2 \ ds. \nonumber\\
\end{eqnarray}
Moreover, if 
 \begin{eqnarray}\label{H1.53}
 \int_0^{+\infty} \left(  \frac{E_2(0)+E^{2p-1}(0)+ \int_0^t h^{2p-1}(s) \ ds}{ \int_0^t \xi^{2p-1}(s) \ ds } \right)^{ \frac{1}{2p-1}} < + \infty,
 \end{eqnarray}
then for all $t \in \mathbb{R}_+$,
 \begin{eqnarray}\label{H1.54}
 E(t) \leq C \left(  \frac{E_2(0)+E^{p}(0)+ \int_0^t h^{p}(s) \ ds}{ \int_0^t \xi^{p}(s) \ ds } \right)^{ \frac{1}{p}}.
 \end{eqnarray}
\end{enumerate}
\end{theorem}
\begin{lemma}
Assume that g satisfies $(\mathbb{H}_2)$ and $(\mathbb{H}_3)$, and u is the solution of \eqref{I1.1}-\eqref{I1.2}. 
Then, for any \  $ 0 < \sigma < 1$ and any \  $t>0$, we have
\begin{eqnarray}\label{IL3.2}
\int_0^{t} g(s) \Vert B^{\frac{1}{2}} \eta^t(s)\Vert^2 \ ds \leq 
C \left[  E(0) \int_0^{t} g^{1-\sigma}(s) \ ds  \right]^{\frac{p-1}{p-1+\sigma} } 
\left[ \int_0^{t} g^p(s)  \Vert B^{\frac{1}{2}} \eta^t(s) \Vert^2 \ ds \right]^{\frac{\sigma}{p-1+\sigma}},\nonumber\\
\end{eqnarray}
and for $ \sigma= \frac{1}{2}$, we have
\begin{eqnarray}\label{IL3.3}
\int_0^{t} g(s) \Vert B^{\frac{1}{2}} \eta^t(s)\Vert^2 \ ds \leq
C \left[  E(0) \int_0^{t} g^{\frac{1}{2}}(s) \ ds  \right]^{\frac{2p-2}{p-1} } 
\left[ \int_0^{t} g^p(s)  \Vert B^{\frac{1}{2}} \eta^t(s) \Vert^2 \ ds \right]^{\frac{1}{2p-1}}.
\end{eqnarray}
\end{lemma}

\begin{proof}
By making use of H\"older's inequality, we obtain
\begin{eqnarray}\label{IL3.5}
\int_0^{t} g(s) \Vert B^{\frac{1}{2}} \eta^t(s)\Vert^2 \ ds
&=& \int_0^{t} g^{1-\frac{\sigma p}{p-1+\sigma}} (s)
\Vert B^{\frac{1}{2}} \eta^t(s) \Vert^{2-\frac{2\sigma}{p-1+\sigma}} \
g(s)^{\frac{p\sigma}{p-1+\sigma}}   \Vert B^{\frac{1}{2}} \eta^t(s) \Vert^{\frac{2\sigma}{p-1+\sigma}}
\ ds \nonumber\\
&\leq& \left(  \int_0^{t} g^{1-\sigma}(s) \Vert B^{\frac{1}{2}} \eta^t(s) \Vert^2  \ ds\right)^{
\frac{p-1}{p-1+\sigma} } \ 
\left( \int_0^{t} g^p(s)   \Vert B^{\frac{1}{2}} \eta^t(s) \Vert^2 \ ds \right)^{
\frac{\sigma}{p-1+\sigma}} .\nonumber\\
\end{eqnarray}

Using $ (\mathbb{H}_1)$, $ (\mathbb{H}_2)$ and the the definition of $E$, we then find that
\begin{eqnarray}\label{ID3.14}
E(t) \geq \frac{1-a_0g_0}{2} \left( \Vert A^{\frac{1}{2}}u(t) \Vert^2 +  \Vert u'(t) \Vert^2+\int_0^{+\infty} g(s) \Vert B^{\frac{1}{2}} \eta^t(s) \ ds  \right), \qquad \forall t\in \mathbb{R}_+. 
\end{eqnarray}

Therefore, from \eqref{IL3.5} and \eqref{ID3.14}, we deduce that
\begin{eqnarray}\label{IL3.6}
\int_0^{t} g(s)^{1-\sigma} \Vert B^{\frac{1}{2}} \eta^t(s) \Vert^2 \ ds 
&\leq&  \int_0^{t} g^{1-\sigma}(s) \Vert  B^{\frac{1}{2}} u(t) + B^{\frac{1}{2}}u(t-s) \Vert^2 \ ds 
\nonumber\\
&\leq& 2 \Vert  B^{\frac{1}{2}} u(t) \Vert^2 \int_0^{+\infty} g^{1-\sigma}(s) \ ds
+ 2 \int_0^t  g^{1-\sigma}(s) \Vert B^{\frac{1}{2}} u(t-s) \Vert^2 \ ds \nonumber \\
 &\leq& 2 a_0 \Vert  A^{\frac{1}{2}} u(t) \Vert^2 \int_0^{+\infty} g^{1-\sigma}(s) \ ds
+ 2 a_0 \int_0^t  g^{1-\sigma}(s) \Vert A^{\frac{1}{2}} u(t-s) \Vert^2 \ ds \nonumber \\
&\leq& C E(0) \int_0^{+\infty} g^{1-\sigma}(s) ds.  
\end{eqnarray} 
Finally, by inserting \eqref{IL3.6} in  \eqref{IL3.5}, we get \eqref{IL3.2}. Inequality \eqref{IL3.3} is simply a particular case of \eqref{IL3.2}.
\end{proof}

\begin{corollary}
Assume that g satisfies $(\mathbb{H}_2)$ and $(\mathbb{H}_3)$, and u is the solution of \eqref{I1.1}-\eqref{I1.2}. Then, for all $t\in \mathbb{R}_+$,
\begin{eqnarray}\label{ID3.9}
\xi(t) \int_0^t g(s)  \Vert B^{\frac{1}{2}} \eta^t(s) \Vert^2 \ ds  \leq C \  [-E'(t) ]^{\frac{1}{2p-1}}.
\end{eqnarray}
\end{corollary}
\begin{proof}
Using  \eqref{IL3.3},  {Lemma2} (for $\sigma=\frac{1}{2}$) and  Young's inequality, we obtain
\begin{eqnarray}\label{ID3.10}
\xi(t) \int_0^t g(s) \Vert B^{\frac{1}{2}} \eta^t(s)\Vert^2 \ ds 
&\leq& C \xi(t)^{\frac{2p-2}{2p-1}}(t)  \left[  \int_0^t g^{\frac{1}{2}}(s) \ ds \right]^{\frac{2p-2}{2p-1}} 
\xi^{\frac{1}{2p-1}}(t)  \left[  \int_0^{ t} g^p(s) \Vert B^{\frac{1}{2}}\eta^t(s) \Vert^2 \ ds  \right]^{\frac{1}{2p-1}}
\nonumber\\
&\leq& C \left[ \int_0^t \xi(s) g^{\frac{1}{2}}(s) \ ds  \right]^{\frac{2p-2}{2p-1}}  
\left[  \int_0^{ t}  \xi(s) g^p(s) \Vert B^{\frac{1}{2}}\eta^t(s) \Vert^2 \ ds  \right]^{\frac{1}{2p-1}}  \nonumber \\
&\leq& C \left[  -\int_0^{ t}  g'(s) \Vert B^{\frac{1}{2}}\eta^t(s) \Vert^2 \ ds  \right]^{\frac{1}{2p-1}} \nonumber \\
&\leq& C [- E'(t) ]^{\frac{1}{2p-1}}.
\end{eqnarray}
\end{proof}
For the following Lemma, we adopt the result from \cite{A11} without proof.
\begin{lemma}
There exist positive constants  $M,\alpha_0, \alpha_1, \alpha_2$ such that the functional 
\begin{eqnarray}\label{ID3.11}
I_3(t)&:=&M E+ I_1+ \frac{g_0}{2}I_2+ \alpha_0 E
\end{eqnarray} 
is equivalent to $ E $ and satisfies, for all $ t \in \mathbb{R}_+$,
\begin{eqnarray}\label{ID3.12}
I_3'(t) \leq -\alpha_1 E(t) +\alpha_{2} \left( \int_0^{+\infty} g(s) \Vert A^{\frac{1}{2}} \eta^t(s) \Vert^2 \ ds
+ \int_0^{+\infty} g(s) \Vert B^{\frac{1}{2}} \eta^t(s) \Vert^2 \ ds \right),
\end{eqnarray}
where the functionals $I_1, I_2$ are given by  
\begin{eqnarray*}
I_1(t):=\langle u_t(t),u(t)\rangle,
\end{eqnarray*}
and 
\begin{eqnarray*}
I_2(t):=-\langle u_t(t), \int_0^{+\infty} g(s) \eta^t(s) \ ds\rangle.
\end{eqnarray*}
\end{lemma}
\begin{proof}
See \cite{A11}
\end{proof}
\begin{proof}(\textbf{Theorem3})

\textbf{Case1: \eqref{H1.4} holds}. \\
 We have 
\begin{eqnarray}\label{ID3.17}
\int_0^{+\infty} g(s) \Vert A^{\frac{1}{2}} \eta^t(s) \Vert^2 \ ds 
\leq a_2  \int_0^{+\infty} g(s) \Vert B^{\frac{1}{2}} \eta^t(s) \Vert^2 \ ds .
\end{eqnarray}
 It then follows  from \eqref{ID3.17} and \eqref{ID3.12}, that, for some positive constant $C$, we get 
\begin{eqnarray}\label{ID3.201}
 I_3'(t) &\leq& -\alpha_1 E(t) + C   \int_0^{+\infty} g(s)\Vert B^{\frac{1}{2}} \eta^t(s) \Vert^2 \ ds.
\end{eqnarray}
Using Corollary 5, we multiply  the estimate \eqref{ID3.201} by  $ \xi(t)$ to arrive at
\begin{eqnarray}\label{ID3.13}
\xi(t) I_3'(t) &\leq& -\alpha_1 \xi(t) E(t) + C \xi(t) \int_0^{+\infty} g(s)\Vert B^{\frac{1}{2}} \eta^t(s) \Vert^2 \ ds.
\end{eqnarray}

Now, from $(\mathbb{H}_1)$ and \eqref{ID3.14} one can see that for all $ s>t $, 
\begin{eqnarray}\label{ID3.15}
\Vert B^{\frac{1}{2}} \eta^t(s) \Vert^2  & \leq& 2a_0 \Vert A^{\frac{1}{2}} u(t) \Vert^2 
+ 2 \Vert B^{\frac{1}{2}} u(t-s) \Vert^2 \nonumber \\
&\leq& \frac{4a_0}{1-a_0g_0} E(0) +  2 \sup_{\tau<0} \Vert B^{\frac{1}{2}} u(\tau) \Vert^2
\nonumber\\
&\leq& \frac{4a_0}{1-a_0g_0} E(0) + 2 m_0^2.
\end{eqnarray}
This leads  to
\begin{eqnarray}\label{ID3.16}
\xi(t) \int_t^{+\infty} g(s) \Vert B^{\frac{1}{2}}  \eta^t(s) \ ds \Vert^2 \ ds 
\leq \left( \frac{4}{1-g_0} E(0) + 2 m_0^2 \right) \xi(t) \int_t^{+\infty} g(s) \ ds :=C h(t).
\end{eqnarray}

From Corollary 5, and \eqref{ID3.16}, the inequality  \eqref{ID3.13} takes the form
\begin{eqnarray}\label{ID3.18}
\xi(t) I_3'(t) &\leq& -\alpha_1 \xi(t) E(t) + C [-E'(t)]^{\frac{1}{2p-1}} + C h(t).
\end{eqnarray}

Multiplying  the last inequality by  $ \xi^\alpha E^\alpha$, where $ \alpha = 2p-2 > 0$, we find
\begin{eqnarray}\label{ID3.19}
\xi(t)^{\alpha+1} E^\alpha(t) I_3'(t) \leq -\alpha_1 \xi^{\alpha+1}(t) E^{\alpha+1}(t)+ C \xi^\alpha(t)E^\alpha(t)[-E'(t)]^{\frac{1}{2p-1}} + C h(t)\xi^\alpha(t)E^\alpha(t) .\nonumber\\
\end{eqnarray}

Exploiting Young's inequality, we get for any $ \epsilon >0$,  
\begin{eqnarray}\label{ID3.20}
\xi^{\alpha+1}(t) E^\alpha(t) I_3'(t) &\leq& -\alpha_1 \xi^{\alpha+1} E^{\alpha+1}(t) - C_\epsilon E'(t)+  2 \epsilon \xi^{\alpha+1}(t)E^{\alpha+1}(t)+ C_\epsilon h^{\alpha+1}(t).\nonumber \\
\end{eqnarray}

Next, let $ F(t) := \xi^{\alpha+1}(t) E^\alpha(t) I_3(t) + C_\epsilon E(t)\sim E(t) $. Then,  for $ \epsilon$ small enough, there exists a positive constant $ \tilde{\alpha}_1 $ such that
\begin{eqnarray}\label{ID3.21}
F'(t) \leq -  \tilde{\alpha}_1 \xi^{\alpha+1}(t) F^{\alpha+1}(t) + C h^{\alpha+1}(t).
\end{eqnarray}

In view of Lemma 1 and taking into account that   $ F \sim E$, we get
\begin{eqnarray}\label{ID3.22}
E(t) \leq C (1+t)^{-\frac{1}{2p-2}} \xi^{-\frac{2p-1}{2p-2}}(t)\left[ 1+ \int_0^t (s+1)^{\frac{1}{2p-2}} \xi^{\frac{2p-1}{2p-2}} (s) h^{2p-1}(s)  \ ds 
 \right]. 
\end{eqnarray} 

To get \eqref{ID3.7}, again we use  estimate \eqref{ID3.201}
\begin{eqnarray*}
\xi(t) I_3'(t) &\leq& -\alpha_1 \xi(t) E(t) + \alpha_2 \xi(t) \int_0^{t} g(s)\Vert A^{\frac{1}{2}} \eta^t(s) \Vert^2 \ ds
 + \alpha_2 \xi(t) \int_t^{+\infty} g(s)\Vert A^{\frac{1}{2}} \eta^t(s) \Vert^2 \ ds.
 \end{eqnarray*}
Applying Jensen's inequality, the estimate \eqref{ID3.6} and the fact that $ \xi $ is non-increasing, we find
\begin{eqnarray}\label{ID3.23}
  \xi(t) \int_0^{t} g(s)\Vert A^{\frac{1}{2}} \eta^t(s) \Vert^2 \ ds 
  &\leq& \frac{\nu(t)}{\nu(t)} \int_0^t [\xi^p(s)g^p(s)]^{\frac{1}{p}} \Vert A^{\frac{1}{2}} \eta^t(s) \Vert^2 \ ds
  \nonumber\\
  &\leq& C \nu(t)\left[ \frac{1}{\nu(t)} \int_0^t \xi^p(s)g^p(s) \Vert A^{\frac{1}{2}} \eta^t(s) \Vert^2 \ ds
 \right]^{\frac{1}{p}} \nonumber \\
 &=& c \nu^{1-\frac{1}{p}}(t)  \left[ \int_0^t \xi^p(s)g^p(s) \Vert A^{\frac{1}{2}} \eta^t(s) \Vert^2 \ ds
 \right]^{\frac{1}{p}},
\end{eqnarray}
where
\begin{eqnarray}\label{ID3.24}
\nu(t) &:=& \int_0^t \Vert A^{\frac{1}{2}}\eta^t(s) \Vert^2 \ ds \leq C \int _0^t \left(  \Vert A^{\frac{1}{2}} u(t) \Vert^2 + \Vert A^{\frac{1}{2}} u(t-s) \Vert^2 \right) \ ds \nonumber \\
&\leq& C \int_0^t [ E(t)+E(t-s) ] \ ds \nonumber \\
&\leq& 2C \int_0^t E(t-s)) \ ds \nonumber \\
&\leq & 2C \int_0^t E(s) \  ds  < 2C \int_0^\infty
E(s) \ ds < \infty,
\end{eqnarray}
and  with the assumption  that $ \nu(t)>0$.\\
Using \eqref{ID3.23} and \eqref{H1.4}, the assumption  $ (\mathbb{H}_3)$ and  the fact that $ \xi $ is non-increasing, we get
\begin{eqnarray}\label{ID3.26}
\xi(t) \int_0^{t} g(s)\Vert A^{\frac{1}{2}} \eta^t(s) \Vert^2 \ ds 
  &\leq& C \nu^{\frac{p-1}{p}}(t) \xi^{p-1}(0)\left[  \int_0^t \xi(s)g^p(s) \Vert B^{\frac{1}{2}} \eta^t(s) \Vert^2 \ ds  \right]^{\frac{1}{p}} \nonumber\\
  &\leq& C \left[  \int_0^t -g'(s) \Vert B^{\frac{1}{2}} \eta^t(s) \Vert^2 \ ds  \right]^{\frac{1}{p}} \nonumber\\
  &\leq& C [-E'(t)]^{\frac{1}{p}}.
 \end{eqnarray} 
 Thus, from \eqref{ID3.16} and \eqref{ID3.26}, it yields that
\begin{eqnarray}\label{ID3.27}
\xi(t) I_3'(t) &\leq& -\alpha_1 \xi(t) E(t) 
+C [-E'(t)]^{\frac{1}{p}} + C h(t).
\end{eqnarray}
We multiply \eqref{ID3.27} by $ \xi^\alpha (t) E^\alpha(t) $, for $ \alpha = p-1 $. This yields
\begin{eqnarray}\label{ID3.28}
\xi^{\alpha+1}(t)E^{\alpha}(t) I_3'(t) +C E'(t) \leq -\beta_1 \xi^{\alpha+1}(t)E^{\alpha+1}(t)+
\beta_2 h^{\alpha+1}(t).
\end{eqnarray} 
Now, let  $  \tilde{F}(t) := \xi^{\alpha+1}(t) E^\alpha(t) I_3(t) + C_\epsilon E(t)\sim E(t) $, then we have
\begin{eqnarray}\label{ID3.29}
\tilde{F}(t) \leq -\beta_1 \xi^{\alpha+1}(t)\tilde{F}^{\alpha+1}(t)+\beta_2 h^{\alpha+1}(t).
\end{eqnarray}
 Then {Lemma1} implies that
\begin{eqnarray}\label{ID3.30}
\tilde{F}(t) \leq C (1+t)^{-\frac{1}{p-1}} \ \xi^{-\frac{p}{p-1}}(t)\left[ 1+  \int_0^t (s+1)^{\frac{1}{p-1}} \xi^{\frac{p}{p-1}} (s) \ h^{p}(s)  \ ds 
 \right]. 
\end{eqnarray}
 Hence, we infer  that
\begin{eqnarray}\label{ID3.31}
E(t)\leq C (1+t)^{-\frac{1}{p-1}} \ \xi^{-\frac{p}{p-1}}(t)\left[ 1+ \int_0^t (s+1)^{\frac{1}{p-1}} \xi^{\frac{p}{p-1}} (s) \ h^{p}(s)  \ ds 
 \right], \quad  \forall t>0.  
\end{eqnarray}
\textbf{Case 2:   \eqref{H1.5} holds}. \\
As in \cite{A11}, and similar to the approach of \cite{Assymp07}, we  recall that the energy $E_2$  related with problem \eqref{I1.1}-\eqref{I1.2} and associated with $\mathcal{A}^{\frac{1}{2}}\mathcal{U}$ corresponding to \ $\mathcal{U}_0 \in\left(  \mathcal{D}(A) \times \mathcal{D}(A^{\frac{1}{2}})\times L^2_g( \mathbb{R}^+, \mathcal{D}(A^{\frac{1}{2}}B^{\frac{1}{2}}))\right)$  defined on
$ \mathbb{R}^+ $  by
\begin{eqnarray*}
E_2(t)&=&\Vert \mathcal{A}^{\frac{1}{2}} \mathcal{U}\Vert^2_{\mathcal{H}} \nonumber\\
&=& \frac{1}{2}\left( \Vert A u(t) \Vert^2-g_0 \Vert A^{\frac{1}{2}}B^{\frac{1}{2}}u(t) \Vert^2
+ \Vert A^{\frac{1}{2}}u'(t)\Vert^2 \right) + \frac{1}{2} \int_0^{+\infty}g(s)\Vert A^{\frac{1}{2}}B^{\frac{1}{2}}\eta^t(s)  \Vert^2 \ ds, \nonumber\\
\end{eqnarray*}
with 
\begin{eqnarray}\label{3.39}
E_2'(t)= \frac{1}{2}\int_0^{+\infty} g'(s) \Vert A^{\frac{1}{2}}B^{\frac{1}{2}}\eta^t(s) \ ds  \leq 0,\qquad \forall
t \in \mathbb{R}_+.
\end{eqnarray}
We observe that in view of assumption $(\mathbb{H}_1)$ 
\begin{eqnarray}\label{H3.40}
\Vert A^{\frac{1}{2}}B^{\frac{1}{2}} v \Vert^2 \leq a_0 \Vert A v  \Vert^2, \qquad \forall v \in \mathcal{D}(A^{\frac{1}{2}}B^{\frac{1}{2}}).
\end{eqnarray}
Multiplying \eqref{ID3.12} by $\xi(t)$, using \eqref{H1.5} and \eqref{H3.40}, we get 
\begin{eqnarray}\label{H3.41}
\xi(t) I_3'(t) &\leq& -\alpha_1 \xi(t) E(t) + \alpha_2 a_2(1+ a_0)\xi(t) \int_0^{+\infty} g(s)\Vert A^{\frac{1}{2}} B^{\frac{1}{2}} \eta^t(s) \Vert^2 \ ds.
\end{eqnarray}
Thanks to $ (\mathbb{H}_2)$, we get
\begin{eqnarray}\label{H3.42}
E_2(t) \geq  \frac{1-g_0 a_0}{2}\left( \Vert A u(t) \Vert^2
+ \Vert A^{\frac{1}{2}}u'(t)\Vert^2 + \int_0^{+\infty}g(s)\Vert A^{\frac{1}{2}}B^{\frac{1}{2}}\eta^t(s)  \Vert^2 \ ds \right)>0. 
\end{eqnarray}
From \eqref{H3.42}, we have
\begin{eqnarray}\label{H3.430}
 \Vert A u(s) \Vert^2
\leq  \frac{2}{1-g_0a_0}E_2(s) \leq
\frac{2}{1-g_0a_0} E_2(0),\ \ \forall s \in \mathbb{R}_+. \nonumber\\
\end{eqnarray}
Therefore, for all $s>t$ 
\begin{eqnarray}\label{H3.440}
\Vert A^{\frac{1}{2}} u(t-s) \Vert \leq  2 \sup_{\tau>0} \Vert A^{\frac{1}{2}} B^{\frac{1}{2}}u_0(\tau) \Vert^2
\leq 2m_0^2.
\end{eqnarray}
Combining  \eqref{H3.430}, \eqref{H3.440}and \eqref{H3.40}, we find
\begin{eqnarray}\label{H3.450}
\xi(t) \int_t^{+\infty} g(s) \Vert A^{\frac{1}{2}} B^{\frac{1}{2}}\eta^t(s) \Vert^2 \
ds &\leq& 2  \xi(t) \int_t^{+\infty} g(s) \Vert A^{\frac{1}{2}} B^{\frac{1}{2}} u(t) \Vert^2 \ ds \nonumber \\
&+& 2  \xi(t) \int_t^{+\infty} g(s) \Vert A^{\frac{1}{2}}B^{\frac{1}{2}} u(t-s) \Vert^2 \ ds 
\nonumber\\
&\leq&  \left(  \frac{4 a_0}{1-g_0a_0} E_2(0)+ 4a_0 m_0^2 \right)\xi(t) \int_t^{+\infty} g(s) \ ds \nonumber\\
&\leq& C h(t). 
\end{eqnarray}

If we repeat the same steps  as in Lemma 4 , and  replace $ A^{\frac{1}{2}}B^{\frac{1}{2}}$ by $A$  in estimate 
\eqref{IL3.6} we end up with
\begin{eqnarray}\label{H3.43}
\xi(t) \int_0^t g(s) \Vert A^{\frac{1}{2}} B^{\frac{1}{2}} \eta^t(s)  \Vert^2  \ ds
&\leq&  C\left[  E_2(0) \int_0^{t} \xi(s) g^{\frac{1}{2}}(s) \ ds  \right]^{\frac{2p-2}{p-1} } 
\left[ \int_0^{t}  \xi(s) g^p(s)  \Vert A^{\frac{1}{2}}B^{\frac{1}{2}} \eta^t(s) \Vert^2 \ ds \right]^{\frac{1}{2p-1}}
\nonumber\\
&\leq& C \left[ \int_0^{t}  \xi(s) g^p(s)  \Vert A^{\frac{1}{2}}B^{\frac{1}{2}} \eta^t(s) \Vert^2 \ ds \right]^{\frac{1}{2p-1}}.
\end{eqnarray}

Using condition \eqref{H1.5}, $(\mathbb{H}_3)$ and combining \eqref{H3.450} and \eqref{H3.43}, we obtain 
\begin{eqnarray}\label{H3.44}
\xi(t)\int_0^{+\infty} g(s) \Vert A^{\frac{1}{2}} \eta^t(s) \Vert^2 \ ds &\leq& a_2
 \xi(t) \int_0^{+\infty} g(s) \Vert A^{\frac{1}{2}}B^{\frac{1}{2}} \eta^t(s) \Vert^2 \ ds \nonumber\\
&\leq& C [-E_2'(t)]^{\frac{1}{2p-1}} + C\ h(t).
\end{eqnarray}

Hence, estimate \eqref{H3.41} will be as
\begin{eqnarray}\label{H3.45}
\xi(t) I_3'(t) \leq -\alpha_1 \xi(t) E(t) + C [-E_2'(t)]^{\frac{1}{2p-1}} + C \ h(t).
\end{eqnarray}
Next, multiplying \eqref{H3.45} by  $ \xi^{\alpha }E^\alpha$ with $ \alpha=2p-2$, and using H\"older's inequality, we get for some $ \tilde{\alpha}_2>0$,
\begin{eqnarray}\label{H3.46}
\xi^{\alpha+1}(t) E^{\alpha}(t) I'_3(t)\leq - \tilde{\alpha}_2 \xi^{\alpha+1}(t)E^{\alpha+1}(t)  - C E'_2(t) 
+C h^{\alpha+1}(t).
\end{eqnarray}
Integrating over $(0,t)$, we find
\begin{eqnarray}\label{H3.47}
\int_0^t \xi^{\alpha+1}(s) E^{\alpha+1}(s) \ ds &\leq& -\frac{1}{\tilde{\alpha}_2} \int_0^t \xi^{\alpha+1}(s) E^{\alpha}(s) I'_3(s) \ ds + \frac{C_\epsilon}{\tilde{\alpha}_2}E_2(0) + \frac{C_\epsilon}{\tilde{\alpha}_2}\int_0^t h^{\alpha+1}(s) \ ds. \nonumber \\
\end{eqnarray}
Now, taking the advantage of the fact that $I_3\sim E$ and  that the energy $ E $ and the function $\xi$ are nonincreasing, we    get
\begin{eqnarray}\label{H3.48}
E^{\alpha+1}(t) \int_0^t \xi^{\alpha+1}(s) \ ds &\leq& - \frac{1}{\tilde{\alpha}_2} \xi^{\alpha+1}(t) E^{\alpha+1}(t)+ \frac{1}{\tilde{\alpha}_2} \xi^{\alpha+1}(0)E^{\alpha+1}(0) + \frac{C}{\tilde{\alpha}_2}E_2(0) + \frac{C}{\tilde{\alpha}_2}\int_0^t h^{\alpha+1}(s) \ ds \nonumber \\
&\leq&  \frac{1}{ \tilde{\alpha}_2} \xi^{\alpha+1}(0)E^{\alpha+1}(0) + \frac{C}{\tilde{\alpha}_2}E_2(0) + \frac{C}{\tilde{\alpha}_2}\int_0^t h^{\alpha+1}(s) \ ds.
\end{eqnarray}
This yields 
\begin{eqnarray}\label{H3.49}
E(t) \leq C \left(  \frac{E_2(0)+E^{\alpha+1}(0)+ \int_0^t h^{\alpha+1}(s) \ ds}{ \int_0^t \xi^{\alpha+1}(s) \ ds } \right)^{ \frac{1}{\alpha+1}}. 
\end{eqnarray}
Therefore,
\begin{eqnarray}\label{H3.50}
E(t) \leq C \left(  \frac{E_2(0)+E^{2p-1}(0)+ \int_0^t h^{2p-1}(s) \ ds}{ \int_0^t \xi^{2p-1}(s) \ ds } \right)^{ \frac{1}{2p-1}}. 
\end{eqnarray}
In order to get \eqref{H1.54}, we use estimate \eqref{H3.42} and follow the same procedure as in estimate \eqref{ID3.26}. We take the operator $ A^{\frac{1}{2}}B^{\frac{1}{2}}$ instead of
$ A^{\frac{1}{2}}$ to find

\begin{eqnarray}\label{H3.51}
\xi(t) \int_0^t g(s) \Vert A^{\frac{1}{2}}B^{\frac{1}{2}} \eta^t(s) \Vert^2 \ ds \leq c [-E'(t) ]^{\frac{-1}{p}}.
\end{eqnarray}
Consequently, from \eqref{H3.450} and \eqref{H3.51}, estimate \eqref{H3.41} will be as
\begin{eqnarray}\label{H3.52}
\xi(t) I_4'(t) &\leq& -\alpha_1 \xi(t) E(t) + C [-E'(t)]^{\frac{-1}{p}}+ C h(t). 
\end{eqnarray}

Finally, by repeating the same steps from \eqref{H3.46} to \eqref{H3.50} with $ \alpha=p-1$, we find   the estimate \eqref{H1.54}, and the proof is now  complete.
\end{proof}

\textbf{Example}:
We illustrate the energy decay rate  given by {Thoerem 3}  through
the following example \\
Let \ $ g(t)=a(1+t)^{-q}, \quad q >2,$ where $a>0$ \ is a constant  so that 
$\int_0^{+\infty} g(t) \ dt < \frac{1}{a}$, then we  have
\begin{eqnarray*}
g'(t)= -a q (1+t)^{- q -1} = - b ( a (1+t)^{-q})^{\frac{q+1}{q}} = -b g^p(t), \quad p=\frac{q+1}{q}, \ b>0.
\end{eqnarray*}

Therefore, for the {Case \eqref{H1.4}}, estimate \eqref{ID3.7} with $ \xi(t)=b$  yields
\begin{eqnarray*}
E(t)\leq C (1+t)^{-\frac{1}{p-1}} \xi^{-\frac{p}{p-1}}(t)\left(  1+  \int_0^t (s+1)^{\frac{1}{p-1}} \xi^{\frac{p}{p-1}} (s) h^{p}(s) \ ds 
 \right). 
\end{eqnarray*}
Let us compute
\begin{eqnarray} 
h(t)=\xi(t) \int_t^{+\infty} g(s) \ ds &=& b \int_t^{+\infty} a (1+s)^{-q} \ ds \nonumber\\
&=& \frac{ab}{q-1} (1+t)^{1-q}, \quad  q=\frac{1}{p-1}.
\end{eqnarray}
Observe that, for some positive constant $C$, it yields
\begin{eqnarray*}
\int_0^t (s+1)^{\frac{1}{p-1}} \xi^{\frac{p}{p-1}} (s) h^{p}(s) \ ds  = b  \frac{ab}{q-1}  \int_0^t (1+s)^{p(1-q)+\frac{1}{p-1}} \ ds = C (1+t)^{p(1-q)+\frac{1}{p-1} +1}-C.  
\end{eqnarray*}
Then, it yields
\begin{eqnarray}\label{H3.53}
E(t)&\leq& C (1+t)^{-\frac{1}{p-1}} \xi^{-\frac{p}{p-1}}(t) \left(  1+C (1+t)^{p(1-q)+\frac{1}{p-1} +1}\right)  \nonumber \\
&=&  C (1+t)^{-\frac{1}{p-1}} + C (1+t)^{p(1-q)+1} \nonumber \\
&=&C (1+t)^{- q} + C(1+t)^{ \frac{- q^2 +q +1}{q}} \leq C(1+t)^{ \frac{- q^2 +q +1}{q}}.
\end{eqnarray}
For the {Case \eqref{H1.5}, estimate \eqref{H1.54} gives
\begin{eqnarray}\label{H3.54}
E(t) &\leq& \left(  \frac{E_2(0)+E^{p}(0)+ \int_0^t h^{p}(s) \ ds}{ \int_0^t \xi^{p}(s) \ ds } \right)^{ \frac{q}{q+1}} \nonumber\\
&\leq& bt^{-\frac{q}{q+1}} \left(  E_2(0)+E^{p}(0)+ \int_0^t h^{p}(s) \ ds \right)^{ \frac{q}{q+1}} \nonumber\\ 
&\leq& C \ t^{-\frac{q}{q+1}}\left( 1-(1+t)^{\frac{(1-q)(q+1)}{q}+1} \right)^{ \frac{q}{q+1}} \leq  C \ t^{-\frac{q}{q+1}}. 
\end{eqnarray} 
Let us compare our result with the one of \cite{A.M14,A11}.  In this way, let us recall the approach  of \cite{A11} with $B=A$, there exists a positive constant $c_1$ such that 
\begin{eqnarray}\label{H3.57}
E(t) \leq  c_1 (1+t)^{-c_2}, \qquad \forall t \in \mathbb{R}_+,
\end{eqnarray}
where $c_2$ is generated by the calculations and it is generally small.\\
Furthermore, the approach of \cite{A11} in polynomial case under $ (A_2)$ and with 
$G(t)=t^{\frac{1}{p}+1}$,  for any $ p \in \left] 0, \frac{q-1}{2} \right[$ gives \\
If  \eqref{H1.4} holds, 
\begin{eqnarray}\label{H3.55}
E(t) \leq C (1+t)^{-p}, \quad \forall p \in \left]0, \frac{q-1}{2} \right[, 
\end{eqnarray} 
and if \eqref{H1.5} holds,
\begin{eqnarray}\label{H3.56}
E(t) \leq C (1+t)^{-\frac{p}{p+1}}, \quad \forall p \in \left]0, \frac{q-1}{2} \right[. 
\end{eqnarray}  
Now, since $ \frac{q^2-q-1}{2}>\frac{q-1}{2}$ for $q>2$. Then from \eqref{H3.53}, \eqref{H3.57} and  \eqref{H3.55} , we conclude that our estimate \eqref{H3.53} gives a  better decay than \eqref{H3.57} and \eqref{H3.55}.\\

For the case \eqref{H1.5}, we see that $ \frac{q}{q+1}> \frac{p}{p+1}$, for any $p \in ]0, \frac{q-1}{2}[$. Then estimate \eqref{H3.54} has better decay than estimate \eqref{H3.56} also for \eqref{H3.53} under some hypothesis on dimension of space. \\

As a conclusion our approach improves and  has a better decay rate  than the one of \cite{A.M14,A11}.\\
\textbf{Acknowledgment:}
The author would like to express his sincere thanks to Professor Aissa Guesmia at University of Lorraine-Metz(France) for his helpful and fruitful comments, and also to the University of Metz during his scholarship.

 \end{document}